\documentclass[12pt]{iopart}

\usepackage{iopams}

\usepackage{graphicx}
\usepackage{amssymb}
\usepackage{amsthm}
\usepackage{latexsym}
\usepackage[colorinlistoftodos, textsize=tiny]{todonotes}
\usepackage{verbatim}
\usepackage{marginnote}

\usepackage[]{hyperref} 		
\hypersetup{
	colorlinks,
	linkcolor=blue, 
	citecolor=blue, 
	linktocpage} 

\theoremstyle{plain}
\newtheorem{theorem}{Theorem}[section]
\newtheorem*{theorem*}{Theorem}

\newtheorem{lemma}[theorem]{Lemma}

\theoremstyle{definition}
\newtheorem{definition}[theorem]{Definition}

\newtheorem{assumption}[theorem]{Assumption}

\def\text#1{\mbox{#1}}

\newcommand{\ra}{\rangle}
\newcommand{\la}{\langle}

\newcommand{\Cop}{\mathcal{C}}

\newcommand{\Mop}{\mathcal{M}}
\newcommand{\Rop}{\mathcal{R}}
\newcommand{\Qop}{\mathcal{Q}}
\newcommand{\Sop}{\mathcal{S}}
\newcommand{\Hsp}{\mathcal{H}}

\newcommand{\dVol}{\text{dV}}
\newcommand{\dS}{\text{dS}}

\usepackage{color}
\usepackage[letterpaper,margin=1.15in]{geometry}

\begin{document}

\title{Photoacoustic imaging taking into account thermodynamic attenuation}

\author{Sebasti\'{a}n Acosta$^{1}$ and Carlos Montalto$^2$}
\address{$^1$ Department of Pediatrics -- Cardiology, Baylor College of Medicine and Texas Children's Hospital, Houston, TX, USA}

\address{$^2$ Department of Mathematics, University of Washington, Seattle, WA, USA}

\eads{\mailto{sacosta@bcm.edu} and \mailto{montcruz@uw.edu}}

\begin{abstract}
In this paper we consider a mathematical model for photoacoustic imaging which takes into account attenuation due to thermodynamic dissipation. The propagation of acoustic (compressional) waves is governed by a scalar wave equation coupled to the heat equation for the excess temperature. We seek to recover the initial acoustic profile from knowledge of acoustic measurements at the boundary. 

We recognize that this inverse problem is a special case of boundary observability for a thermoelastic system. This leads to the use of control/observability tools to prove the unique and stable recovery of the initial acoustic profile in the weak thermoelastic coupling regime.
This approach is constructive, yielding a solvable equation for the unknown acoustic profile. Moreover, the solution to this reconstruction equation can be approximated numerically using the conjugate gradient method. If certain geometrical conditions for the wave speed are satisfied, this approach is well--suited for variable media and for measurements on a subset of the boundary. We also present a numerical implementation of the proposed reconstruction algorithm.
\end{abstract}


\submitto{\IP}
% \maketitle

%\listoftodos

%%%%%%%%%%%%%%%%%%%%%%%%%%%%%%%%%%%%%%%%%%%%%%%%%%%%%
%%%%%%% NEW SECTION %%%%%%%%%%%%%%%%%%%%%%%%%%%%%%%%%
%%%%%%%%%%%%%%%%%%%%%%%%%%%%%%%%%%%%%%%%%%%%%%%%%%%%%
\section{Introduction} \label{Section:Intro}

Photoacoustic tomography is an imaging technique that takes advantage of the high--contrast exhibited by optical absorption and the high--resolution carried by broadband acoustic waves in soft biological tissues. Details concerning this type of imaging modalities are found in \cite{Wang-Wu-2007,Wang-2009,Beard-2011,Cox-Laufer-Beard-2009, Anastasio-2011, Wang-2012,Wang-2003, Zhang-2007, Shan-2008, Pramanik-Wang-2009, Cox2012}. Qualitative photoacoustic imaging consists of recovering an initial pressure profile from acoustic measurements acquired on the boundary of a region of interest. The successful transformation of boundary measurements into the sought interior pressure profile requires mathematical algorithms that have been studied by numerous researchers. Some of them are based on explicit formulas valid for waves propagating in free--space and homogeneous media \cite{Finch2004,Kunyansky2007,Finch2007,Kunyansky2011,WangAnastasio2012,Natterer2012,Palamodov2012,Haltmeier2014}. Others seek to account for variable wave speed and/or the presence of boundaries \cite{Anastasio-2007, Hristova-et-al-2008, Hristova2009, Ste-Uhl-2009-01, Stefanov-Uhlmann-2011,kunyansky-holman-cox-2013, Tittelfitz-2012, HuangWaNie2013, WangXia2014, holman-kunyansky-2014, Acosta-Montalto-2015,stefanov-yang-2015,Nguyen2015,ChervovaOksanen2016,StefanovYang2016}. See also the reviews \cite{Agranovsy-2007,Kuchment-2008,Kuchment2011,Bal-2011} for additional references.

There have been recent efforts to incorporate acoustic attenuation in the modeling of photoacoustic tomography (PAT) and thermoacoustic tomography (TAT). See \cite{LaRiviere2009,Kowar-2010,Treeby2010b,Kowar-Scherzer-2012,Ammari-Bretin-2012,Treeby-2010,Roitner-2011,Cook-2011,Huang-2012,
Kalimeris-2013,Homan-2013,Kowar-2014,Palacios2016} and references therein. Most of these results aim at modeling attenuation in the frequency domain to account for dissipation and dispersion. In this paper, however, we adopt a model where the propagation of acoustic waves is thermodynamically coupled to the diffusion of heat. The photoacoustic effect, on which PAT is physically based, consists of two transformations of energy. First, electromagnetic energy is absorbed and transformed into heat. Second, there is a conversion of heat into mechanical energy due to thermal expansivity of the tissues. Concerning this second step, due to the thermodynamic interaction between temperature and pressure, the reverse transformation of energy also takes place. Since heat diffuses, this process attenuates the energy of the thermoacoustic waves. See \cite[Ch. 8]{Athanasiou-Book-2009} for an introduction to thermoelasticity in biomechanics. We claim that this type of attenuation should be naturally considered in PAT because PAT itself is based on the thermo--elastic interaction.

We realize that, in mathematical terms, the PAT problem coincides with a problem of \textit{boundary observability} --- the ability to determine the solution of a partial differential equation from knowledge of overdetermined boundary data (Dirichlet and Neumann).
This is one of the central concepts of control theory for partial differential equations \cite{Glow-Lions-He-2008,Bar-Leb-Rau-1992,GLLT-2004,Tucsnak-Weiss-Book-2009}. We have already employed similar tools to address the PAT problem in an enclosure \cite{Acosta-Montalto-2015} and other related problems \cite{Aco-2013,Acosta2015301}. The objective of this paper is to constructively employ the tools of observability for hyperbolic equations together with certain regularity properties of parabolic equations to solve the PAT problem in the presence of thermodynamic attenuation. For the thermoelastic system, there is a series of works on establishing exact, approximate and null controllability or observability estimates \cite{Zuazua-1994,deTeresa-1996,Lebeau-Zuazua-1998,Liu-1998,Liu-1998-2,Liu-Williams-1998,Eller2000,Albano-Tataru-2000,GLLT-2004}. In these works, however, either the boundary condition or the distribution of control/observation is not of the type we need to model the PAT problem. Therefore, we modify some ideas provided by these references to seek a solution for PAT in the weak coupling regime. Although it might be possible to use Carleman estimates from \cite{Eller2000,Albano-Tataru-2000,GLLT-2004} to treat the strongly coupled system, we refrain from doing so because the thermoelastic coupling in PAT is known to be relatively weak. Hence, we claim that the results of this paper are sufficient for the nature of PAT in biological tissues.

%%%%%%%%%%%%%%%%%%%%%%%%%%%%%%%%%%%%%%%%%%%%%%%%%%%%%%%%%%%%%%%%%
%%%%%%% NEW SECTION %%%%%%%%%%%%%%%%%%%%%%%%%%%%%%%%%%%%%%%%%%%%%
%%%%%%%%%%%%%%%%%%%%%%%%%%%%%%%%%%%%%%%%%%%%%%%%%%%%%%%%%%%%%%%%%
\section{Mathematical Formulation and Main Results} \label{Section:Formulation}

In this paper we study the photoacoustic tomography problem in the presence of thermodynamic dissipation. This is modeled by the linear equations of elasticity coupled with thermal diffusivity \cite{Athanasiou-Book-2009}. Let $\Omega \subset \mathbb{R}^{n}$ be a smooth bounded connected domain with boundary $\partial \Omega$ where $n \geq 2$. The propagation of thermoelastic waves in isotropic media is governed by the following system \cite[Ch. 8]{Athanasiou-Book-2009},
\begin{eqnarray*}
\fl \rho \partial_{t}^{2} \textbf{u} - \nabla \left( \lambda \, \text{div} \, \textbf{u} \right) - \text{div} \, \mu \left( \nabla \textbf{u} + \nabla \textbf{u}^{\rm T} \right) + \beta K \, \nabla \theta = 0, \qquad && \mbox{in} \quad (0,\tau) \times \Omega \label{Eqn:Displacement}  \\
\fl \partial_{t} \theta - \alpha \Delta \theta + \frac{\theta_{\rm ref} \beta K} {\rho c_{\rm p}}   \, \text{div} \, \partial_{t} \textbf{u} = 0, \qquad && \mbox{in} \quad (0,\tau) \times \Omega \label{Eqn:Temp} 
\end{eqnarray*}
for the displacement $\textbf{u}$ and where $\theta$ denotes the deviation from the reference temperature $\theta_{\rm ref}$. Also, $\rho$ is the mass density, $\lambda$ and $\mu$ are the Lam\'{e} coefficients, $\alpha$ denotes the thermal diffusivity and $c_{\rm p}$ is the specific heat at constant pressure. The thermoelastic coupling is given by $\beta K$ where $\beta$ is the coefficient of volumetric thermal expansion and $K$ is the bulk modulus. In soft biological tissue, $\lambda \gg \mu$ which implies that $K / (\lambda + 2 \mu) \approx 1$. For the moment let us assume that $\lambda$, $\mu$, $K$ and $\rho$ are constants in $\Omega$. Later, we will drop this assumption.

\begin{table}[!h]
\caption{\label{Table:Param} Physical parameters for the thermoacoustic coupling in typical soft biological tissues.}
\begin{indented}
\item[]\begin{tabular}{@{}llll}
\br
Physical Parameter & Symbol & Range & Units   \\
\mr
Bulk modulus & $K$ & 2000 -- 2500 & $10^{6}$ Pa \\
Density & $\rho$   & 900 -- 1100  & Kg / m$^3$ \\
Ref. Temperature  & $\theta_{\rm ref}$   & 290 -- 310 & K \\
Coeff. thermal expansion & $\beta$ & 200 -- 300 & $10^{-6}$ / K \\
Specific heat & $c_{\rm p}$ & 500 -- 5000 & J / (Kg K) \\
\br
\end{tabular}
\end{indented}
\end{table}

Since photoacoustic imaging is primarily concerned with the compressional waves, we define the pressure $p = - (\lambda + 2 \mu) \, \text{div} \, \textbf{u}$ and the square of the compressional wave speed $c^2 = (\lambda + 2 \mu) / \rho$ and proceed to obtain a scalar model for the thermoacoustic waves. Simultaneously, we seek to reveal the strength of the thermoelastic coupling by writing the governing equations in unitless form. Let $L$ be a characteristic length scale of the domain $\Omega$ (such as its diameter). Let the characteristic time scale be given by $T = L / c_{\rm ref}$ where $c^{2}_{\rm ref} = K/\rho$ is the square of a reference wave speed. We define the following unitless variables and parameters:
\begin{itemize}
\item[] Pressure $\hat{p} = p / K$ and temperature $\hat{\theta} = \theta / \theta_{\rm ref}$.
\item[] Length $\hat{x} = x / L$ and time $\hat{t} = t / T$. 
\item[] Parameters: $\hat{c}^2 = c^2 T^2 / L^2$, $\hat{\alpha} = \alpha T / L^2$, $\sigma = K / ( \theta_{\rm ref} \rho c_{\rm p})$ and $\epsilon = \beta \theta_{\rm ref}$.
\end{itemize}
The unitless coupling parameter $\epsilon > 0$ is introduced to analyze the case where $\epsilon$ is sufficiently small. This is valid for a small coefficient of thermal expansion $\beta$. The unitless product $\mathcal{G} = \epsilon \sigma$ is known as the Gr\"{u}neisen coefficient. Table \ref{Table:Param} 
displays rough estimates for the values of these physical parameters for soft biological tissues. We obtain that $0.05 \lesssim \epsilon \lesssim 0.1$ and $0.5 \lesssim \sigma \lesssim 10$.

For notational convenience, we assume that $\sigma = 1$. This presents no impediment to the theory as we could easily treat the case $\sigma > 0$. At this point, in order to alleviate the notation, we also drop the caret to denote the unitless quantities. The unitless scalar governing system then becomes, 
\begin{eqnarray}
\partial_{t}^{2} p - c^2 \Delta p - \epsilon \, c^2 \Delta \theta = 0 \qquad && \mbox{in} \quad (0,\tau) \times \Omega, \label{Eqn:Pressure01}  \\
\partial_{t} \theta - \alpha \Delta \theta - \epsilon \, \partial_{t} p = 0 \qquad && \mbox{in} \quad (0,\tau) \times \Omega. \label{Eqn:Temp01} 
\end{eqnarray}
We take (\ref{Eqn:Pressure01})--(\ref{Eqn:Temp01}) as the starting point for the modeling of this problem. While the equivalence of the thermoelastic system and equations (\ref{Eqn:Pressure01})--(\ref{Eqn:Temp01}) only holds when $\lambda$, $\mu$, $K$ and $\rho$ are constants, we will consider a positive wave speed $c \in C^2(\overline{\Omega})$. This follows the common practice of considering a variable wave speed to model heterogeneous media, even when the wave equation is not in divergence form. Similarly, we assume a positive thermal diffusivity $\alpha \in C^2(\overline{\Omega})$. The system (\ref{Eqn:Pressure01})--(\ref{Eqn:Temp01}) is augmented by the following initial and boundary conditions,
\begin{eqnarray}
p = p_{0}, \quad  \partial_{t} p = p_{1}  \quad \text{and} \quad \theta = \theta_{0} \qquad && \mbox{on} \quad \{ t = 0 \} \times \Omega, \label{Eqn:InitialCond01} \\
\partial_{\nu} p + \gamma \partial_{t} p  = 0 \quad \text{and} \quad \partial_{\nu} \theta = 0 \qquad && \mbox{on} \quad (0,\tau) \times \partial \Omega. \label{Eqn:BdyCond01}
\end{eqnarray}
Here, $\gamma :\partial \Omega \to [0,\infty)$ denotes the acoustic impedance coefficient at the boundary $\partial \Omega$. We assume that $\gamma \in C^2(\partial \Omega)$. Physically, $\gamma = 0$ models an acoustically hard surface (such as reflectors) and $\gamma \to \infty$ approximates an acoustically soft boundary. In general we allow $\gamma$ to vary on the boundary $\partial \Omega$ to model the heterogeneous nature of an enclosing surface and the interface with sensors or air. 
The length of the observation window of time is given by $\tau < \infty$ which is defined below. In (\ref{Eqn:BdyCond01}), the symbol $\partial_{\nu}$ denotes the outward normal derivative at the boundary $\partial \Omega$. 

Concerning the initial conditions, it is common in the modeling of photoacoustic tomography to assume the following \cite{XiaYaoWang2014}.
\begin{assumption}[\textbf{Rapid Deposition of Heat}] \label{Assump.FastHeatDepo} 
The initial conditions (\ref{Eqn:InitialCond01}) satisfy, 
\begin{eqnarray*}
p_{1} = 0 \qquad \text{and} \qquad \theta_{0} = \epsilon \, p_{0}. 
\end{eqnarray*}
\end{assumption}

These two conditions are respectively known as \textit{stress confinement} and \textit{thermal confinement} \cite{XiaYaoWang2014}. They are valid when the pressure relaxation and thermal diffusion are negligible in the very short lapse of heat deposition from the optical source. These conditions can be achieved in biological tissues by using nanosecond optical pulses \cite{Wang-Wu-2007,Wang-2009,Beard-2011,Cox-Laufer-Beard-2009,
Anastasio-2011,Wang-2012,XiaYaoWang2014,Kruger-1999,Larina-2005}. The assumption $\theta_{0} = \epsilon p_{0}$ is mathematically crucial because it removes an important degree of freedom in the analysis. It would not be possible to recover, in stable manner, an independent initial condition for the thermal field. This is a well--known consequence of the smoothing effect of the heat equation. See details in \cite{Zuazua-1994,deTeresa-1996,Lebeau-Zuazua-1998,Liu-Williams-1998,Eller2000,Albano-Tataru-2000}.

In order to consider partial measurements, we divide the boundary as the disjoint union $\partial \Omega = \Gamma \cup (\partial \Omega \setminus \Gamma)$ where $\Gamma$ is the portion where we make observations of the acoustic field. We also assume that $\{ x \in \partial \Omega : \gamma(x) > 0 \} \subset \Gamma$ so that the absorptive part of the boundary (where $\gamma > 0$) is contained within the observable part of the boundary. As reviewed in the next section, the forward problem (\ref{Eqn:Pressure01})--(\ref{Eqn:BdyCond01}) has a unique solution, and we can define the \textit{measurement} map given by
\begin{eqnarray}
\Mop p_{0} = p|_{(0,\tau) \times \Gamma} \label{Eqn:MeasMap}
\end{eqnarray}
where $p$ is the solution of (\ref{Eqn:Pressure01})--(\ref{Eqn:BdyCond01}) with initial conditions satisfying Assumption \ref{Assump.FastHeatDepo}.
The goal of the photoacoustic tomography problem is to find the initial profile $p_{0}$ from knowledge of $\Mop p_{0}$. This is a challenging problem with intricate dependencies between the domain $\Omega$, the partial boundary $\Gamma$, the wave speed $c$ and the time interval $(0,\tau)$. The admissible dependencies are made precise by a sophisticated assumption of geometric character. Following Bardos, Lebeau and Rauch \cite{Bar-Leb-Rau-1992}, we assume that our problem enjoys the \textit{geometric control condition} for the Riemannian manifold $(\Omega, c^{-2} dx^2)$ with only the portion $\Gamma$ of the boundary $\partial \Omega$ being accessible for observation. 
We assume that $\Gamma$ is a smooth open domain relative to $\partial \Omega$ and that all the geodesics of $(\Omega, c^{-2} dx^2)$ have finite contact order with the boundary $\partial \Omega$. Under this condition, the geodesic rays of $(\Omega, c^{-2} dx^2)$ can be uniquely extended when they encounter the boundary $\partial \Omega$. See mathematical details in \cite{Bar-Leb-Rau-1992}. The geometric assumption needed for our main result is the following.

\begin{assumption}[\textbf{Geometric Condition}] \label{Assump.001} 
There exists $\tau_{\rm o} < \infty$ such that any geodesic ray of the manifold $(\Omega, c^{-2} dx^2)$, originating from any point in $\Omega$ at $t=0$, eventually reaches $\Gamma$ at a non--diffractive point (after possible geometrical reflections on $\partial \Omega \setminus \Gamma$) before time $t=\tau_{\rm o}$. Also assume that $\tau > \tau_{\rm o}$.
\end{assumption}

A geodesic ray is non--diffractive if, in the absence of the boundary, the ray leaves $\Omega$. See the precise mathematical definition in \cite{Bar-Leb-Rau-1992}. In physical terms, Assumption \ref{Assump.001} means that acoustic signals have an interaction with the boundary $\Gamma$ strong enough for all acoustic signals to deliver a non--negligible amount of their energy to the boundary.

The main theoretical result of this paper is that even though the thermodynamic attenuation affects the acoustic waves, the pressure measurements acquired on the boundary $(0,\tau) \times \Gamma$ are sufficient to stably recover the initial state of the pressure field (provided that the thermoelastic coupling is sufficiently weak). We make this statement precise in the form of a theorem.
\begin{theorem}[\textbf{Main Result}] \label{Thm.MainResult}
Suppose that Assumptions \ref{Assump.FastHeatDepo} and \ref{Assump.001} hold. There exists $\epsilon_{\rm o} > 0$ so that if $0 \leq \epsilon < \epsilon_{\rm o}$, the map $p_{0} \mapsto \Mop p_{0}$ is injective and satisfies a stability estimate of the form
\begin{eqnarray*}
\|  p_{0} \|_{H^{1}(\Omega)} \leq C \|  \Mop p_{0} \|_{H^{1}((0,\tau) \times \Gamma)},
\end{eqnarray*}
for all $p_{0} \in H^{1}(\Omega)$, and some positive constant $C=C(\epsilon_{\rm o})$ independent of $\epsilon$.
\end{theorem}

By contrast to existing results for the observability/controllability of solutions to thermoelastic equations, we highlight that the conclusion of  Theorem \ref{Thm.MainResult} is true without having to observe the temperature field $\theta$ on the boundary $\Gamma$. However, this novel result depends critically on the thermal confinement assumption --- that the initial temperature profile is proportional to the initial pressure profile. See Assumption \ref{Assump.FastHeatDepo}. 

The proof of Theorem \ref{Thm.MainResult} is presented in Section \ref{Section:Proof}.
In addition to this theoretical result, we also propose a convergent iterative reconstruction algorithm which is described in Section \ref{Section:Reconst}.

%%%%%%%%%%%%%%%%%%%%%%%%%%%%%%%%%%%%%%%%%%%%%%%%%%
%%%%%%% NEW SECTION %%%%%%%%%%%%%%%%%%%%%%%%%%%%%%
%%%%%%%%%%%%%%%%%%%%%%%%%%%%%%%%%%%%%%%%%%%%%%%%%%
\section{Proof of the Main Result} \label{Section:Proof}

In order to properly analyze the inverse problem we must first state some mathematical properties of the initial boundary value problem (\ref{Eqn:Pressure01})--(\ref{Eqn:BdyCond01}). Our guiding references are \cite{Lebeau-Zuazua-1998,Lio-Mag-Book-1972,EvansPDE,Ren-Rog-2004,Eng-Nag-2000,Dafermos1968,Lebeau1997}. We denote by $H^k(\Omega)$ for $k\in \mathbb{Z}$ the Sobolev space of order $k$ over $L^2(\Omega)$. Notice, $H^0(\Omega) = L^2(\Omega)$. See \cite[\S 5.2 -- \S 5.9]{EvansPDE} for an introduction to Sobolev spaces as well as the Bochner spaces $H^{j}((0,\tau); H^{k}(\Omega))$ and $C^{j}((0,\tau); H^{k}(\Omega))$ which employ in this Section. We use the following definition of energy for the thermoacoustic system,
\begin{eqnarray*}
 E(t) = \frac{1}{2} \int_{\Omega} \left( |\nabla p(t,x)|^2 + c^{-2}(x) |\partial_{t} p(t,x)|^2 + |\nabla \theta(t,x)|^2  \right) \dVol(x).
\end{eqnarray*}
for any triplet $(p(t), \partial_t p(t), \theta(t)) \in H^1(\Omega) \times H^{0}(\Omega) \times H^1(\Omega)$. Notice that any pair $p = \text{const}$ and $\theta = \text{const}$ is a solution of (\ref{Eqn:Pressure01}), (\ref{Eqn:Temp01}) and (\ref{Eqn:BdyCond01}) with zero energy. There are infinitely many nonzero constant solutions with vanishing energy which implies that the energy norm does not identify solutions uniquely. This can be remedied by considering only solutions that satisfy the following conditions
\begin{eqnarray}
\int_{\Omega} c^{-2}(x) \partial_{t} p(t,x) \, \dVol(x) + \int_{\partial \Omega} \gamma(x) p(t,x) \, \dS(x) = 0, \label{Eqn:AvgCond01} \\
\int_{\Omega}  \left( \theta(t,x) - \epsilon p(t,x) \right) \, \dVol(x) = 0. \label{Eqn:AvgCond02}
\end{eqnarray}
This is motivated by the fact that the left--hand sides of (\ref{Eqn:AvgCond01})--(\ref{Eqn:AvgCond02}) are indeed independent of time provided that $p$ and $\theta$ solve the governing equations (\ref{Eqn:Pressure01})--(\ref{Eqn:Temp01}) with boundary conditions (\ref{Eqn:BdyCond01}). Therefore, it is only needed to require (\ref{Eqn:AvgCond01})--(\ref{Eqn:AvgCond02}) at time $t=0$. Then the appropriate energy space is given by
\begin{eqnarray*}
\Hsp := \left\{ (p_{0},p_{1},\theta_{0}) \in H^{1}(\Omega) \times H^{0}(\Omega) \times H^{1}(\Omega) : \text{(\ref{Eqn:AvgCond01})--(\ref{Eqn:AvgCond02}) are satisfied}  \right\}.
\end{eqnarray*}
Notice that $\Hsp$ is a closed subspace of $H^{1}(\Omega) \times H^{0}(\Omega) \times H^{1}(\Omega)$. So it is complete under the norm of $H^{1}(\Omega) \times H^{0}(\Omega) \times H^{1}(\Omega)$ as well as under the energy norm. 
We seek a weak solution $(p,\theta)$ of (\ref{Eqn:Pressure01})--(\ref{Eqn:BdyCond01}) in the space $\Hsp$. We say that the functions $p \in H^{k}((0,\tau);H^{1-k}(\Omega))$ for $k=0,1,2$ and $\theta \in H^{j}((0,\tau);H^{1-j}(\Omega))$ for $j=0,1$ are a \textit{weak} solution of (\ref{Eqn:Pressure01})--(\ref{Eqn:BdyCond01}) provided that
\begin{eqnarray*}
\la c^{-2} \partial^{2}_{t} p(t) , v \ra_{\Omega} + \la \nabla p(t) , \nabla v \ra_{\Omega} + \epsilon \la \nabla \theta(t) , \nabla v \ra_{\Omega} + \la \gamma \partial_{t} p(t) , v \ra_{\partial \Omega} = 0, \\
\la \alpha^{-1} \partial_{t} \theta(t) , \vartheta \ra_{\Omega} + \la \nabla \theta(t) , \nabla \vartheta \ra_{\Omega} - \epsilon \la \alpha^{-1} p(t) , \vartheta \ra_{\Omega} = 0,
\end{eqnarray*}
for all $v,\vartheta \in H^{1}(\Omega)$ and $t \in (0,\tau)$, such that $(p(0),\partial_{t}p(0),\theta(0)) = (p_{0},p_{1},\theta_{0}) \in \Hsp$.

The problem (\ref{Eqn:Pressure01})--(\ref{Eqn:BdyCond01}) is well--posed on the space $\Hsp$ and the energy is non--increasing. The proof follows from the standard analysis of partial differential equations and semigroup theory \cite{Lio-Mag-Book-1972,EvansPDE,Ren-Rog-2004,Eng-Nag-2000,Lebeau-Zuazua-1998}. The well--posedness can be established using energy estimates (see details in \cite[Ch 7]{EvansPDE} and \cite{Dafermos1968}) or by expressing the governing equations (\ref{Eqn:Pressure01})--(\ref{Eqn:Temp01}) as a system with first--order time derivatives to analyze the spectral properties of the corresponding infinitesimal generator for the strongly continuous semigroup. 
The details of this latter approach for the thermoacoustic system (\ref{Eqn:Pressure01})--(\ref{Eqn:Temp01}) are found in \cite[\S 2]{Lebeau-Zuazua-1998} and \cite[Ch 2--3]{Tucsnak-Weiss-Book-2009}. We state this in the form of a lemma.

\begin{lemma} \label{Lemma.Forward}
Given $(p_{0},p_{1},\theta_{0}) \in \Hsp$, the unique weak solution of (\ref{Eqn:Pressure01})--(\ref{Eqn:BdyCond01}) satisfies $(p,\partial_{t} p,\theta) \in \Hsp$ for $t \geq 0$, and $p \in C^{k}([0,\tau] ; H^{1-k}(\Omega))$ for $k=0,1$ and $\theta \in C([0,\tau];H^{1}(\Omega))$. 
Moreover, $E(t) \leq E(0)$ for all $t \geq 0$ and all $\epsilon \geq 0$. In fact (due to parabolic regularity) the energy $E \in H^{1}(0,\tau)$ and
\begin{eqnarray*}
 \frac{d E}{dt}(t) = -  \int_{\Omega} \alpha |\Delta \theta(t,x)|^2 \dVol(x) - \int_{\partial \Omega} \gamma(x) |\partial_{t} p(t,x)|^2 \dS(x).
\end{eqnarray*}
\end{lemma}

The well--posedness and regularity implications of Lemma \ref{Lemma.Forward} also apply to arbitrary initial conditions $(p_{0},p_{1},\theta_{0}) \in H^{1}(\Omega) \times H^{0}(\Omega) \times H^{1}(\Omega)$. By virtue of linearity, we can decompose the initial conditions as follows,
\begin{eqnarray*}
p_{0} = \left( p_{0} - p_{0,\rm const} \right) + p_{0,\rm const} \\
\theta_{0} = \left( \theta_{0} - \theta_{0,\rm const} \right) + \theta_{0,\rm const}
\end{eqnarray*}
where 
\begin{eqnarray*}
\fl p_{0,\rm const} =  \left( \int_{\partial \Omega} \gamma(x) \dS(x) \right)^{-1} \left( \int_{\Omega} c^{-2}(x) p_{1}(x) \dVol(x) + \int_{\partial \Omega} \gamma(x) p_{0}(x) \dS(x) \right) \\ 
\fl \theta_{0,\rm const} =  \frac{1}{|\Omega|} \int_{\Omega} \left( \theta_{0}(x) - \epsilon p_{0}(x) \right) \dVol(x) + \epsilon p_{0,\rm const}.
\end{eqnarray*}
Therefore, the evolution of the triplet $(p_{0},p_{1},\theta_{0}) \in H^{1}(\Omega) \times H^{0}(\Omega) \times H^{1}(\Omega)$ can be decomposed into the evolution of the initial condition $(p_{0}-p_{0,\rm const},p_{1},\theta_{0}-\theta_{0,\rm const})$ in $\Hsp$ plus a time--independent solution given by $(p_{0,\rm const},0,\theta_{0,\rm const})$ for $t \geq 0$. Notice that the energy of this particular solution is zero because it is constant both in space and time. This leads to the following result using Lemma \ref{Lemma.Forward} and the decomposition described above.

\begin{theorem}[\textbf{Forward well--posedness}] \label{Thm.Forward}
Given $(p_{0},p_{1},\theta_{0}) \in H^{1}(\Omega) \times H^{0}(\Omega) \times H^{1}(\Omega)$, the unique weak solution of (\ref{Eqn:Pressure01})--(\ref{Eqn:BdyCond01}) satisfies $p \in C^{k}([0,\tau] ; H^{1-k}(\Omega))$ for $k=0,1$ and $\theta \in C([0,\tau];H^{1}(\Omega))$. 
Moreover, $E(t) \leq E(0)$ for all $t \geq 0$ and all $\epsilon \geq 0$. 
\end{theorem}

Now we proceed to re--state and prove Theorem \ref{Thm.MainResult}.

\begin{theorem*}[\textbf{Main Result}] 
Suppose that Assumptions \ref{Assump.FastHeatDepo} and \ref{Assump.001} hold. There exists $\epsilon_{\rm o} > 0$ so that if $0 \leq \epsilon < \epsilon_{\rm o}$, the map $p_{0} \mapsto \Mop p_{0}$ is injective and satisfies a stability estimate of the form
\begin{eqnarray*}
\|  p_{0} \|_{H^{1}(\Omega)} \leq C \|  \Mop p_{0} \|_{H^{1}((0,\tau) \times \Gamma)},
\end{eqnarray*}
for all $p_{0} \in H^{1}(\Omega)$, and some positive constant $C=C(\epsilon_{\rm o})$ independent of $\epsilon$.
\end{theorem*}

\begin{proof}[\textbf{Proof of Theorem \ref{Thm.MainResult}}]
Consider the problem (\ref{Eqn:Pressure01})--(\ref{Eqn:BdyCond01}) for $p_{0} \in H^{1}(\Omega)$ and $p_{1} = 0$ and $\theta_{0} = \epsilon p_{0}$ (Assumption \ref{Assump.FastHeatDepo}). From Theorem \ref{Thm.Forward} we have that the solution to this problem satisfies $p \in C^{k}([0,\tau] ; H^{1-k}(\Omega))$ for $k=0,1$ and $\theta \in C([0,\tau];H^{1}(\Omega))$. Now notice that $\theta$ has initial condition in $H^{1}(\Omega)$, homogeneous Neumann boundary condition, and a forcing term $\epsilon \partial_{t} p \in C([0,\tau];H^{0}(\Omega)) \subset H^{0}((0,\tau);H^{0}(\Omega))$. Hence, from regularity theory for the parabolic equation (see \cite[Thm 4.3, \S 4, Ch 4, Vol II]{Lio-Mag-Book-1972} or \cite[\S 7.1.3]{EvansPDE}), we obtain that $\theta \in H^{0}((0,\tau);H^{2}(\Omega)) \cap H^{1}((0,\tau);H^{0}(\Omega))$ and an estimate of the form
\begin{eqnarray}
 \| \Delta \theta \|^{2}_{H^{0}((0,\tau) \times \Omega)} \leq \epsilon^2 C \left( \|  \partial_{t} p \|^{2}_{H^{0}((0,\tau) \times \Omega)} + \| \nabla p_{0} \|^{2}_{H^{0}(\Omega)} \right), \label{Eqn.HeatEstimate}
\end{eqnarray}
for some constant $C>0$ independent of $\epsilon \geq 0$. From the energy estimate in Theorem \ref{Thm.Forward}, we also have that
\begin{eqnarray}
 \| \partial_{t} p \|^{2}_{H^{0}((0,\tau) \times \Omega)} \leq C (1 + \epsilon^2) \| \nabla p_{0} \|^{2}_{H^{0}(\Omega)}. \label{Eqn.EnergyEstimate}
\end{eqnarray}
Now, under the geometric condition in Assumption \ref{Assump.001}, the acoustic problem for $p$ (governed by the wave equation (\ref{Eqn:Pressure01}) with source term $\epsilon c^2 \Delta \theta$) enjoys the following observability property (see details in \cite[Ch 7]{Tucsnak-Weiss-Book-2009}, \cite[Ch 6]{Glow-Lions-He-2008}, \cite[Lemma 3.3]{Alabau-Boussouira2013a} and \cite{Bar-Leb-Rau-1992,GLLT-2004}),
\begin{eqnarray}
\| p_{0} \|^{2}_{H^{1}(\Omega)} \leq C \left( \epsilon^2 \| \Delta \theta \|^{2}_{H^{0}((0,\tau) \times \Omega)} + \| p \|^{2}_{H^{1}((0,\tau) \times \Gamma)} \right), \label{Eqn.WaveEstimate}
\end{eqnarray}
where $C>0$ is also independent of $\epsilon \geq 0$. 
Combining the above three inequalities, we  obtain that
\begin{eqnarray*}
\| p_{0} \|^{2}_{H^{1}(\Omega)}  \leq C \left(  \| p \|^{2}_{H^{1}((0,\tau) \times \Gamma)} + \epsilon^4 (2 + \epsilon^2) \| \nabla p_{0} \|^{2}_{H^{0}(\Omega)} \right).
\end{eqnarray*}
Therefore, we select $\epsilon_{\rm o}$ so that $C \epsilon_{\rm o}^4 (2 + \epsilon_{\rm o}^2) < 1$. Then for any $0 \leq \epsilon < \epsilon_{\rm o}$, the second term on the right--hand side of the above inequality can be absorbed into the left--hand side to obtain the desired estimate. This concludes the proof.
\end{proof}

%%%%%%%%%%%%%%%%%%%%%%%%%%%%%%%%%%%%%%%%%%%%%%%%%%%%%%%%%%%%%%
%%%%%%% NEW SECTION %%%%%%%%%%%%%%%%%%%%%%%%%%%%%%%%%%%%%%%%%%%%%
%%%%%%%%%%%%%%%%%%%%%%%%%%%%%%%%%%%%%%%%%%%%%%%%%%%%%%%%%%%%%%%%%

\section{Reconstruction Algorithm} \label{Section:Reconst}

In this section we explicitly recover the initial acoustic profile $p_{0}$ in terms of the boundary measurements $\Mop p_{0} = p|_{(0,\tau) \times \Gamma}$. This is accomplished by using Theorem \ref{Thm.MainResult} obtained in the previous section which leads to the invertibility of the normal operator $(\Mop^{*}\Mop)$.

In order to obtain an applicable expression for the operator $\Mop^{*}$, in this section we state the dual or adjoint problem associated with (\ref{Eqn:Pressure01})--(\ref{Eqn:BdyCond01}). This is equivalent to constructing the well--known Hilbert Uniqueness Method (HUM) for control of partial differential equations. See \cite{Lions-Review-1988,Glow-Lions-He-2008} for an overview of these ideas and their historical origin. Throughout, we assume that $\epsilon > 0$ is sufficiently small for Theorem \ref{Thm.MainResult} to apply. This adjoint problem is to find a solution $(\psi,\xi)$ (defined by transposition as in \cite[Ch 3, \S 9]{Lio-Mag-Book-1972} or \cite[\S 4]{Bar-Leb-Rau-1992}) for the following IBVP,
\begin{eqnarray}
\partial_{t}^{2} \psi - c^2 \Delta \psi - \epsilon \, c^2 \alpha^{-1}\partial_{t} \xi = 0 \qquad && \mbox{in} \quad (0,\tau) \times \Omega, \label{Eqn:Adj_Pressure01}  \\
\partial_{t} \xi + \alpha \Delta \xi - \epsilon \, \alpha \Delta \psi = 0 \qquad && \mbox{in} \quad (0,\tau) \times \Omega, \label{Eqn:Adj_Temp01} \\
\psi = 0, \quad  \partial_{t} \psi = 0  \quad \text{and} \quad \xi = 0 \qquad && \mbox{on} \quad \{ t = \tau \} \times \Omega, \label{Eqn:Adj_InitialCond01} \\
\partial_{\nu} \psi - \gamma \partial_{t} \psi  = \eta \quad \text{and} \quad \partial_{\nu} \xi - \epsilon \partial_{\nu} \psi = 0 \qquad && \mbox{on} \quad (0,\tau) \times \partial \Omega. \label{Eqn:Adj_BdyCond01}
\end{eqnarray}
for a given $\eta \in H^{-1}((0,\tau) \times \Gamma)$ (extended as zero on $(0,\tau) \times \partial \Omega \setminus \Gamma$). Notice that this problem is solved backwards in time with vanishing Cauchy data at time $t = \tau$ and that the signs of the terms $\partial_{t} \xi$ and $\alpha \Delta \xi$ are consistent with solving the heat equation backwards in time in a stable manner. In fact, the well--posedness of the dual system (\ref{Eqn:Adj_Pressure01})--(\ref{Eqn:Adj_BdyCond01}) is equivalent to the well--posedness of the primal system (\ref{Eqn:Pressure01})--(\ref{Eqn:BdyCond01}). See \cite[Ch 3, \S 9]{Lio-Mag-Book-1972}, \cite[\S 2.8]{Tucsnak-Weiss-Book-2009} and \cite{Bar-Leb-Rau-1992,Lions-Review-1988} for details. We obtain the following definition.

\begin{definition} \label{Def:SolutionAdj}
Let $\Sop$ be the mapping $\eta \mapsto - \partial_{t} \psi |_{t=0}$, where $\psi$ is the solution of (\ref{Eqn:Adj_Pressure01})--(\ref{Eqn:Adj_BdyCond01}) for the provided $\eta$.
\end{definition}
Integrating by parts the terms of equations (\ref{Eqn:Pressure01})--(\ref{Eqn:Temp01}) against $(\psi,\xi)$, where $(\psi,\xi)$ is the solution of (\ref{Eqn:Adj_Pressure01})--(\ref{Eqn:Adj_BdyCond01}), we easily obtain that
\begin{eqnarray*}
\la p_{0} , \Sop \eta \ra_{\Omega} =  \la \Mop p_0 , \eta \ra_{(0,\tau) \times \Gamma }, \quad \text{for all $\eta \in H^{-1}((0,\tau) \times \Gamma)$ and $p_{0} \in H^{1}(\Omega)$.}
\end{eqnarray*}
Hence, by definition we have that $\Mop^{*} = \Rop \Sop \Qop^{-1} : H^{1}((0,\tau) \times \Gamma) \to H^{1}(\Omega)$ where $\Rop : H^{-1}(\Omega) \to H^{1}(\Omega)$ and $\Qop : H^{-1}((0,\tau) \times\Gamma) \to H^{1}((0,\tau) \times\Gamma)$ are the Riesz representation unitary operators. Now, if we choose $\eta = \Qop^{-1} \Mop p_{0}$ and use the estimate from Theorem \ref{Thm.MainResult}, we obtain that,
\begin{eqnarray*}
 \la  p_{0} , (\Mop^{*} \Mop) p_{0} \ra_{H^{1}(\Omega)} = \| \Mop p_{0} \|^{2}_{H^{1}((0,\tau)\times \Gamma)} \geq C \| p_{0} \|^{2}_{H^{1}(\Omega)},
\end{eqnarray*}
for all $p_{0} \in H^{1}(\Omega)$ and some constant $C>0$. Therefore, the operator $\Mop^{*} : H^{1}((0,\tau) \times \Gamma) \to H^{1}(\Omega)$ is surjective and $(\Mop^{*} \Mop) : H^{1}(\Omega) \to H^{1}(\Omega)$ is coercive. With these results, we can establish the following controllability theorem.

\begin{theorem}[\textbf{Acoustic Control}] \label{Thm.Control}
Let the Geometric Condition \ref{Assump.001} hold. For sufficiently small $\epsilon > 0$, the operator $\Sop : H^{-1}((0,\tau) \times \Gamma) \to  H^{-1}(\Omega)$ given in Definition \ref{Def:SolutionAdj} is surjective. Therefore, for any $\phi \in H^{-1}(\Omega)$, there exists a boundary control $\eta \in H^{-1}((0,\tau) \times \Gamma)$ such that the solution $(\psi, \xi)$ of (\ref{Eqn:Adj_Pressure01})--(\ref{Eqn:Adj_BdyCond01}) satisfies
\begin{eqnarray*}
\partial_{t}\psi = -\phi, \qquad \text{at time $t=0$.}
\end{eqnarray*}
Among all such boundary controls, there exists $\eta_{\rm min}$ which is uniquely determined by $\phi$ as the minimum norm control and satisfies the following stability condition
\begin{eqnarray*}
\| \eta_{\rm min} \|_{H^{-1}((0,\tau) \times \Gamma)} \leq C \| \phi \|_{H^{-1}(\Omega)}
\end{eqnarray*}
for some constant $C > 0$. As a consequence, the mapping $\phi \mapsto \eta_{\rm min}$ defines a bounded control operator $\Cop : H^{-1}(\Omega) \to H^{-1}((0,\tau) \times \Gamma)$, that satisfies $\Cop = \Sop^{*}(\Sop \Sop^{*})^{-1}$. It also follows that $\Qop \Cop \Rop^{-1} = \Mop (\Mop^{*} \Mop)^{-1}$.
\end{theorem}

Let $(\psi,\xi)$ be the solution of (\ref{Eqn:Adj_Pressure01})--(\ref{Eqn:Adj_BdyCond01}) with $\eta = \Cop \phi$ and $\phi \in H^{-1}(\Omega)$ arbitrary. Then by construction, 
\begin{eqnarray*}
\la p_{0} , \phi \ra_{H^{1}(\Omega) \times H^{-1}(\Omega)} =  \la  \Mop p_0 , \Cop \phi \ra_{H^{1}((0,\tau) \times \Gamma) \times H^{-1}((0,\tau) \times \Gamma)},
\end{eqnarray*}
for all $\phi \in H^{-1}(\Omega)$, which implies that the unknown initial condition $p_{0}$ is explicitly recovered as follows,
\begin{eqnarray}
p_{0} = \Cop^{*} \Mop p_0, \label{Eqn.Solution}
\end{eqnarray}
where $\Cop^{*} : H^{1}((0,\tau) \times \Gamma) \to H^{1}(\Omega)$ is the adjoint of the control operator $\Cop : H^{-1}(\Omega) \to H^{-1}((0,\tau) \times \Gamma)$ defined in Theorem \ref{Thm.Control}. 
The reconstruction algorithm is based on the identity (\ref{Eqn.Solution}) and an iterative algorithm to approximate the action of $\Cop^{*}$. This algorithm is based on the following points provided by Theorem \ref{Thm.Control} (cf. \cite{Lions-Review-1988,Glow-Lions-He-2008}):
\begin{enumerate}
\item[(1)] The observability operator $\Cop^{*} =  (\Mop^{*} \Mop)^{-1} \Mop^{*} $, where $(\Mop^{*} \Mop) : H^{1}(\Omega) \to H^{1}(\Omega)$ is coercive.
\item[(2)] For $\zeta \in H^{1}(\Omega)$, the solution to $(\Mop^{*} \Mop) \phi = \zeta$ can be approximated using the conjugate gradient method.
\end{enumerate}

Now we proceed to describe how the action of $\Cop^{*}$ can be approximated using the conjugate gradient method. See \cite[\S 4.6]{Atkinson-Han-Book-2001} for a standard description of the conjugate gradient method in a Hilbert space setting. For sake of completeness, we describe the inversion of a generic equation $(\Mop^{*} \Mop) \phi = \zeta$. Let $\phi_{0}$ be an initial guess for the true solution $\phi_{*}$. Define $r_{0} = \zeta - (\Mop^{*} \Mop) \phi_{0}$ as the initial residue and $s_{0} = r_{0}$. For $k \geq 0$, define
\begin{eqnarray*}
\phi_{k+1} = \phi_{k} + \alpha_{k} s_{k}, & \qquad \text{where} \quad \alpha_{k} = \frac{\| r_{k} \|^{2}_{H^{1}(\Omega)}  }{ \la  s_{k} , (\Mop^{*} \Mop) s_{k} \ra_{H^{1}(\Omega)}  } & \\
r_{k+1} = \zeta - (\Mop^{*} \Mop) \phi_{k+1} \\
s_{k+1} = r_{k+1} + \beta_{k} s_{k}, & \qquad \text{where} \quad \beta_{k} = \frac{\| r_{k+1} \|^{2}_{H^{1}(\Omega)} }{\| r_{k} \|^{2}_{H^{1}(\Omega)} }.
\end{eqnarray*}
Since the operator $(\Mop^{*} \Mop) : H^{1}(\Omega) \to H^{1}(\Omega)$ is bounded and coercive, then there are positive constants $m$ and $M$ such that
\begin{eqnarray*}
m \| \phi \|^{2}_{H^{1}(\Omega)} \leq \la  \phi , (\Mop^{*} \Mop) \phi \ra_{H^{1}(\Omega)} \leq M \| \phi \|^{2}_{H^{1}(\Omega)}.
\end{eqnarray*}
The conjugate gradient iterates can be shown to converge as follows (see \cite[\S 4.6]{Atkinson-Han-Book-2001} and references therein),
\begin{eqnarray}
\fl \| \phi_{*} - \phi_{k} \|_{H^{1}(\Omega)} \leq e^{ - \sigma k} \| \phi_{*} - \phi_{0} \|_{H^{1}(\Omega)}, \qquad \text{for $k \geq 0$, where} \quad \sigma = \ln \left( \frac{M+m}{M-m} \right). \label{Eqn:ConvergenceCG}
\end{eqnarray}

Notice that at each iteration, one must apply the operator $(\Mop^{*} \Mop)$ which amounts to solve the problem (\ref{Eqn:Pressure01})--(\ref{Eqn:BdyCond01}) (under Assumption \ref{Assump.FastHeatDepo}) followed by solving the adjoint problem (\ref{Eqn:Adj_Pressure01})--(\ref{Eqn:Adj_BdyCond01}). In practice, this can be approximated using numerical methods for PDEs. However, depending on the method of choice, there are intrinsic complications that may prevent a convergence estimate such as (\ref{Eqn:ConvergenceCG}) from being satisfied in the limit as the discretization is refined. We shall not elaborate any further on these complications as they lie outside of the scope of this paper. For details on these numerical issues we refer to \cite{Glow-Lions-He-2008,Ervedoza-Zuazua-2010,Ervedoza-2009,Asch-Lebeau-1998,Ervedoza-Zuazua-Book-2013} and references therein. In this paper, we adopted the two--grid approach described in \cite{Asch-Lebeau-1998,Glow-Lions-He-2008} using second order finite difference methods. For the two--grid approach, recall that the computation of residual $r_{k+1} = \zeta - (\Mop^{*} \Mop) \phi_{k+1}$ is understood in the $H^{1}(\Omega)$--sense which means that $r_{k+1}$ solves the equation $\la \nabla (\zeta - r_{k+1}) , \nabla v \ra_{\Omega} = F(v)$ for all $v \in H_{0}^{1}(\Omega)$ where $F = (\Mop^{*} \Mop) \phi_{k+1}$ acts as a functional. This elliptic equation is solved using a grid that is coarser than the grid employed to propagate the wave fields. This computation on a coarser grid has a filtering effect which removes high--frequency oscillations from the residual. In turn, this procedure regulates the convergence of the algorithm as the grids are refined \cite{Asch-Lebeau-1998,Glow-Lions-He-2008}.

%%%%%%%%%%%%%%%%%%%%%%%%%%%%%%%%%%%%%%%%%%%%%%%%%%%%%%%%%%%%%%%%%%
%%%%%%%% NEW SECTION %%%%%%%%%%%%%%%%%%%%%%%%%%%%%%%%%%%%%%%%%%%%%
%%%%%%%%%%%%%%%%%%%%%%%%%%%%%%%%%%%%%%%%%%%%%%%%%%%%%%%%%%%%%%%%%%

\section{Numerical Results} \label{Section:Numerics}
Now we present some numerical results to illustrate the performance of the reconstruction algorithm described in Section \ref{Section:Reconst}. We implemented a numerical solver for the governing system (\ref{Eqn:Pressure01})--(\ref{Eqn:BdyCond01}) and its adjoint (\ref{Eqn:Adj_Pressure01})--(\ref{Eqn:Adj_BdyCond01}) based on second order finite differences. To avoid spurious numerical instabilities, we adopted the two--grid approach described in \cite{Asch-Lebeau-1998,Glow-Lions-He-2008}. We worked in $\mathbb{R}^2$ where the domain $\Omega$ was taken as the unit--square. The initial profile $p_{0}$ corresponds the Shepp--Logan phantom.

We present two examples. One with constant wave speed $c(x) \equiv 1$, and the other with variable wave speed $c=c(x)$ defined below. In both cases, we used the following parameters: impedance $\gamma(x) = c^{-1}(x)$ over the boundary of $\Omega$, thermal diffusivity $\alpha = 0.01$, and coupling parameter $\epsilon = 0.1$. The observability time was chosen to be $\tau = 2$ which is enough for more than $99 \%$ of the energy contained in the initial profile to dissipate or leave the domain through the boundary when the wave speed is constant.

We shall compare the results from the proposed algorithm against the results from purely acoustic time--reversal. The latter is accomplished by producing measurements using the thermoacoustic forward solver $\Mop$, and then back--propagating the boundary measurements in a purely acoustic medium ($\epsilon = 0$), that is, by ignoring the thermodynamic attenuation. See details in \cite{Acosta-Montalto-2015,Ste-Uhl-2009-01} for the purely acoustic time--reversal approach. The acoustic time--reversal is approximated using the same finite difference method. The initial guess for the conjugate gradient algorithm is the approximate solution obtained from the purely acoustic time--reversal algorithm. Although the proposed reconstruction algorithm has been described in the $H^{1}(\Omega)$ setting, a similar study could be performed in the $H^{0}(\Omega)$ setting where the inner--products in the conjugate gradient algorithm would need to be understood appropriately. In this section, we present results from the implementation both in the $H^{1}(\Omega)$ and $H^{0}(\Omega)$ formulations.

\subsection{Constant wave speed}
For the first example where $c \equiv 1$, Figure \ref{Fig.Comparison01} displays the exact initial profile and the reconstructions. The relative errors in the $H^{1}(\Omega)$ and $H^{0}(\Omega)$ formulations are reported in Table \ref{Table.Ex1} for the first few iterations of the conjugate gradient algorithm. We notice that by ignoring the thermodynamic attenuation in the purely acoustic time--reversal reconstruction (Iter $=0$), the edges in the Shepp--Logan phantom are blurred considerably. Some of the sharpness is recovered by accounting for the attenuation in the proposed algorithm even after a single iteration.

\begin{table}[h]
\centering
\caption{\label{Table.Ex1} Constant wave speed example. Relative error at each iteration of the conjugate gradient method described in Section \ref{Section:Reconst}. Iter $=0$ corresponds to the initial guess given by a purely acoustic time--reversal algorithm.}
\begin{tabular}{@{}crr}
\br
Iter & $H^{1}(\Omega)$--norm & $H^{0}(\Omega)$--norm \\
\mr
0 & 52.6 \% & 31.1 \% \\ 
1 & 19.8 \% & 12.8 \% \\ 
2 & 10.6 \% & 5.7 \% \\ 
3 & 6.3 \% &  4.4 \% \\ 
4 & 4.5 \% & 3.8 \% \\ 
5 & 3.8 \% & 3.1 \% \\ 
\br
\end{tabular}
\end{table}

\begin{figure}[h]
\centering
\includegraphics[height=0.33 \textheight, trim=80 10 10 10]{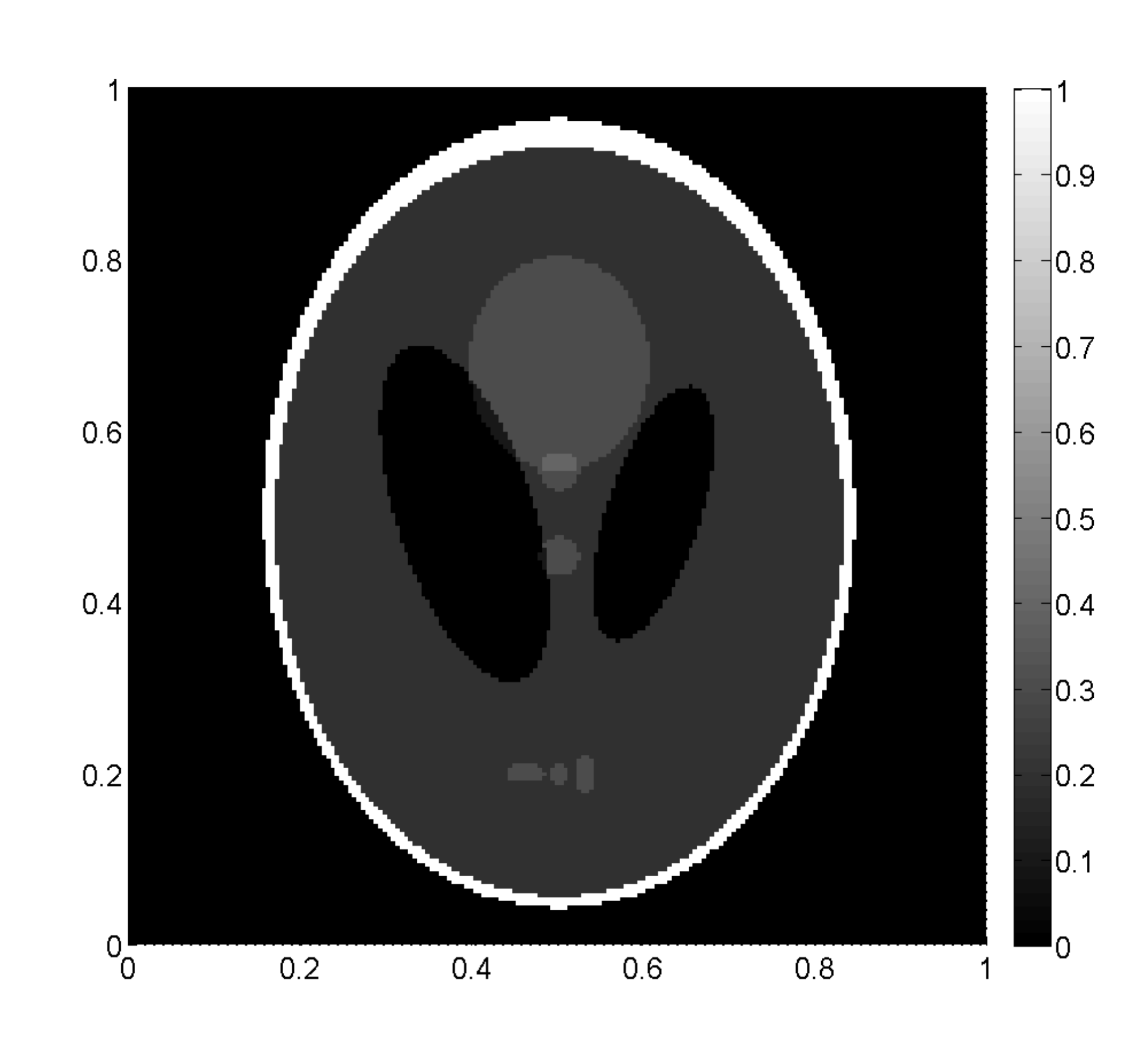} 
\includegraphics[height=0.33 \textheight, trim=10 10 50 10]{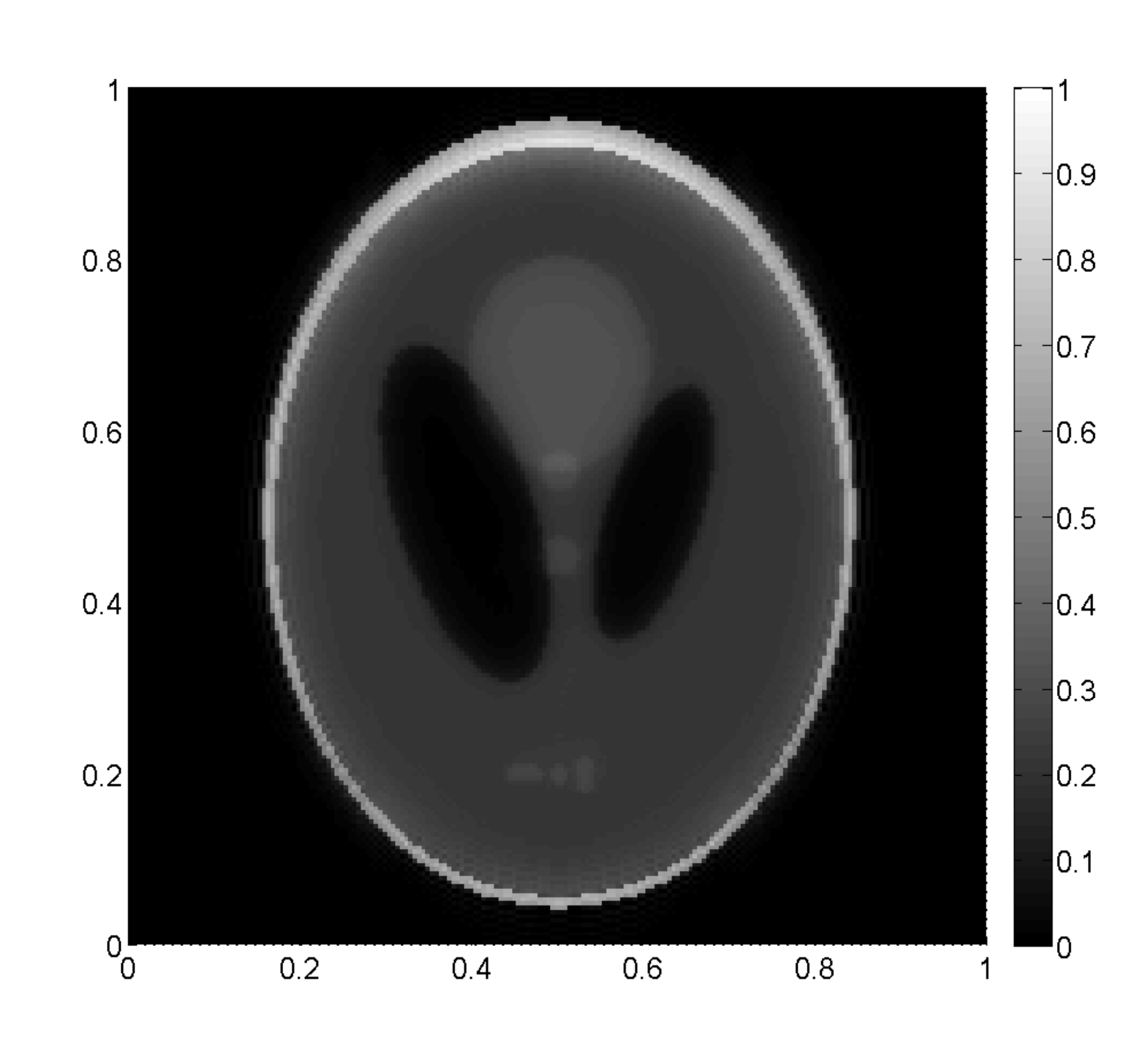} \\
\includegraphics[height=0.33 \textheight, trim=80 10 10 10]{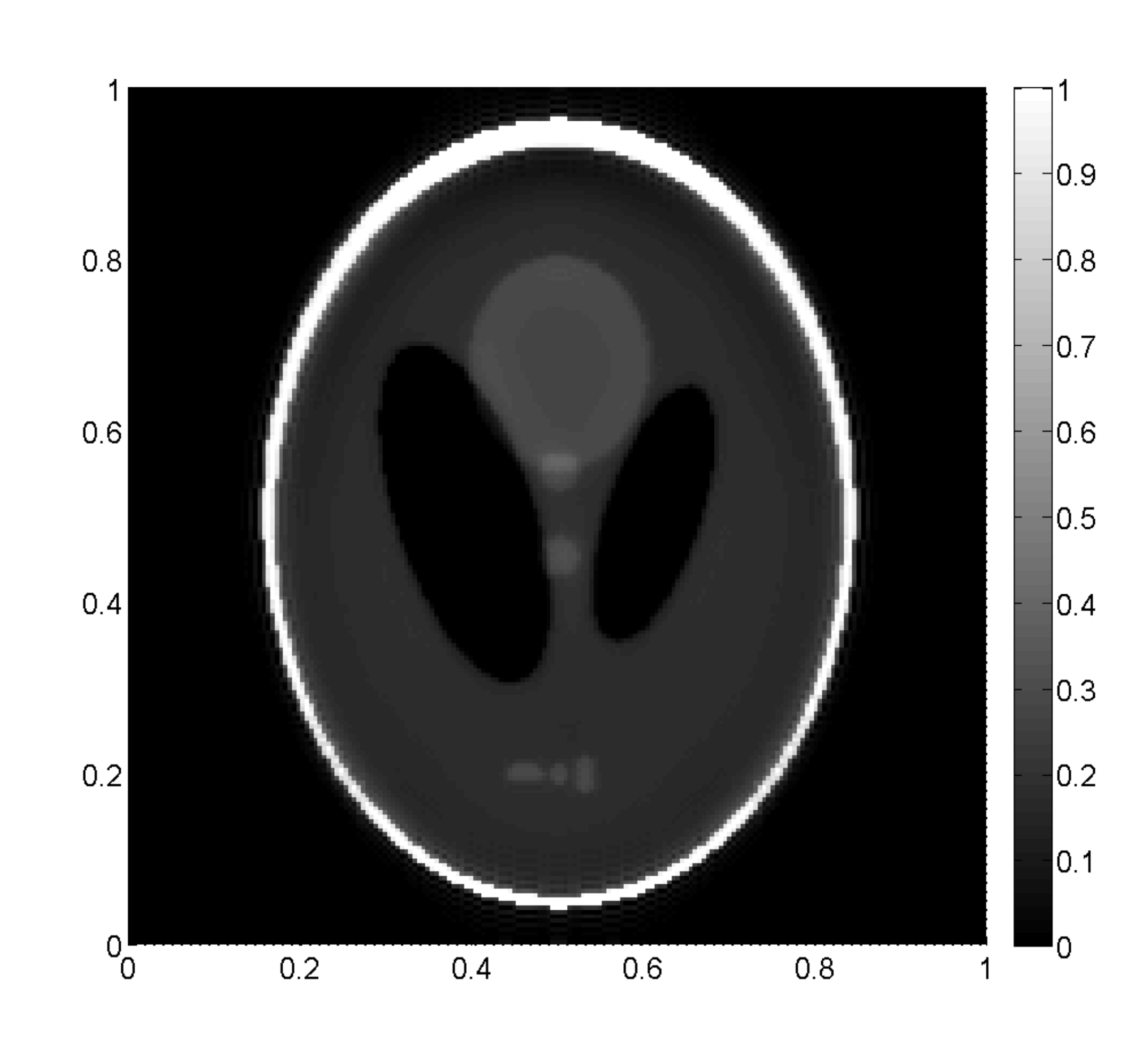} 
\includegraphics[height=0.33 \textheight, trim=10 10 50 10]{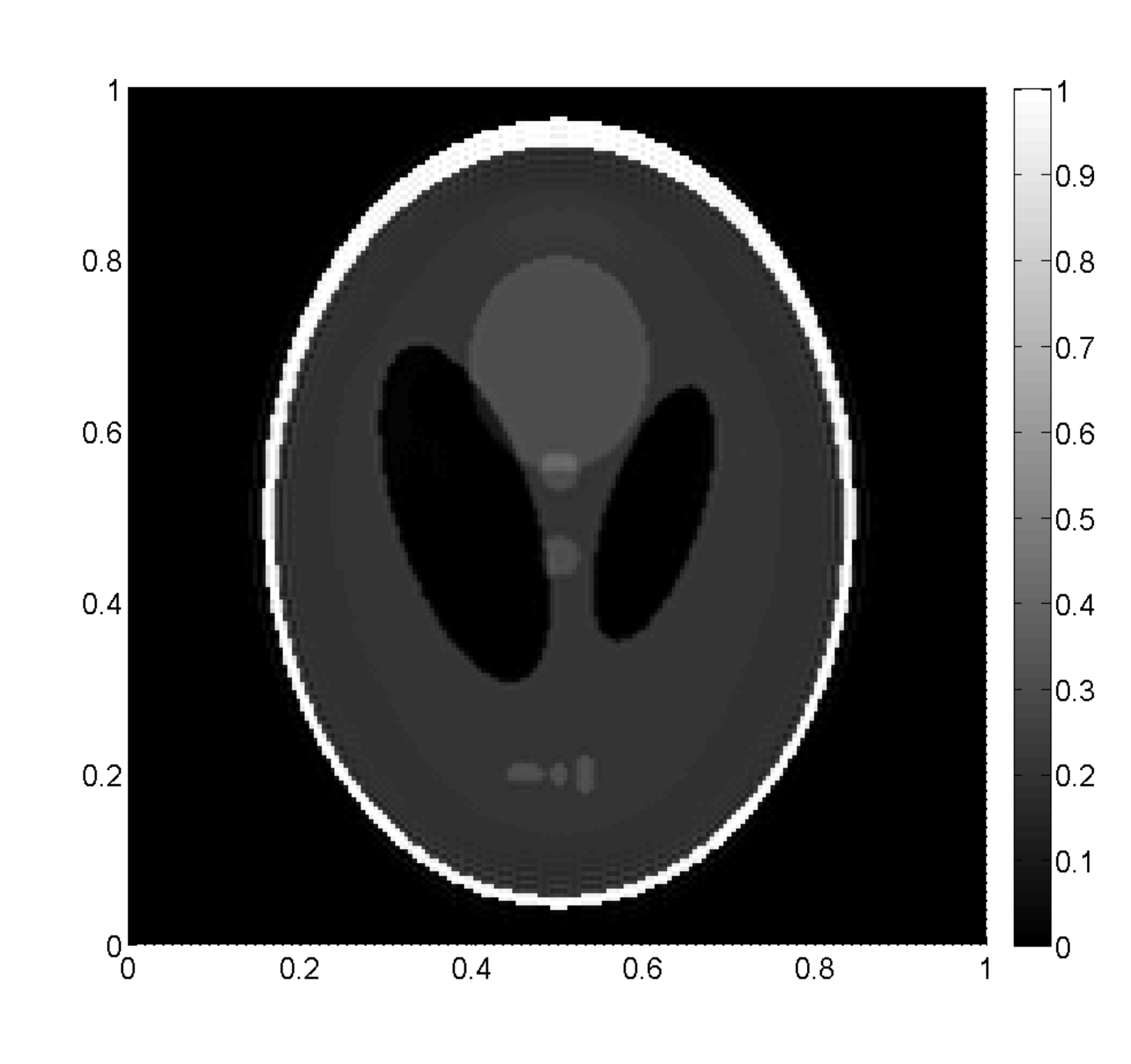} 
\caption{Exact initial acoustic profile (top--left), the reconstruction from purely acoustic time-reversal (top--right), and the reconstruction from the proposed algorithm described in Section \ref{Section:Reconst} using $1$ iteration (bottom--left) and $5$ iterations (bottom--right).}
\label{Fig.Comparison01}
\end{figure}

\subsection{Variable wave speed}
For the second example we have selected a variable wave speed defined as a layer of higher speed surrounding the smaller ellipses in the Shepp--Logan phantom. The actual profile is illustrated in the top--right panel of Figure \ref{Fig.Comparison02}. The relative errors in the $H^{1}(\Omega)$ and $H^{0}(\Omega)$ formulations are reported in Table \ref{Table.Ex2} for the first few iterations of the conjugate gradient algorithm. Again, as shown in the lower panels of Figure \ref{Fig.Comparison02}, we see great improvements over the purely acoustic time--reversal reconstruction. We highlight the ability in capturing the jump discontinuities and the reduction of the artifacts introduced by ignoring the attenuation.

\begin{table}[h]
\centering
\caption{\label{Table.Ex2} Variable wave speed example. Relative error at each iteration of the conjugate gradient method described in Section \ref{Section:Reconst}. Iter $=0$ corresponds to the initial guess given by a purely acoustic time--reversal algorithm.}
\begin{tabular}{@{}crr}
\br
Iter & $H^{1}(\Omega)$--norm & $H^{0}(\Omega)$--norm \\
\mr
0 & 55.4 \% & 34.3 \% \\ 
1 & 24.2 \% & 16.6 \% \\ 
2 & 15.5 \% & 8.0 \% \\ 
3 & 11.8 \% & 4.7 \% \\ 
4 & 10.2 \% & 4.0 \% \\ 
5 & 9.7 \% & 3.7 \% \\ 
\br
\end{tabular}
\end{table}

\begin{figure}
\centering
\includegraphics[height=0.33 \textheight, trim=80 10 10 10]{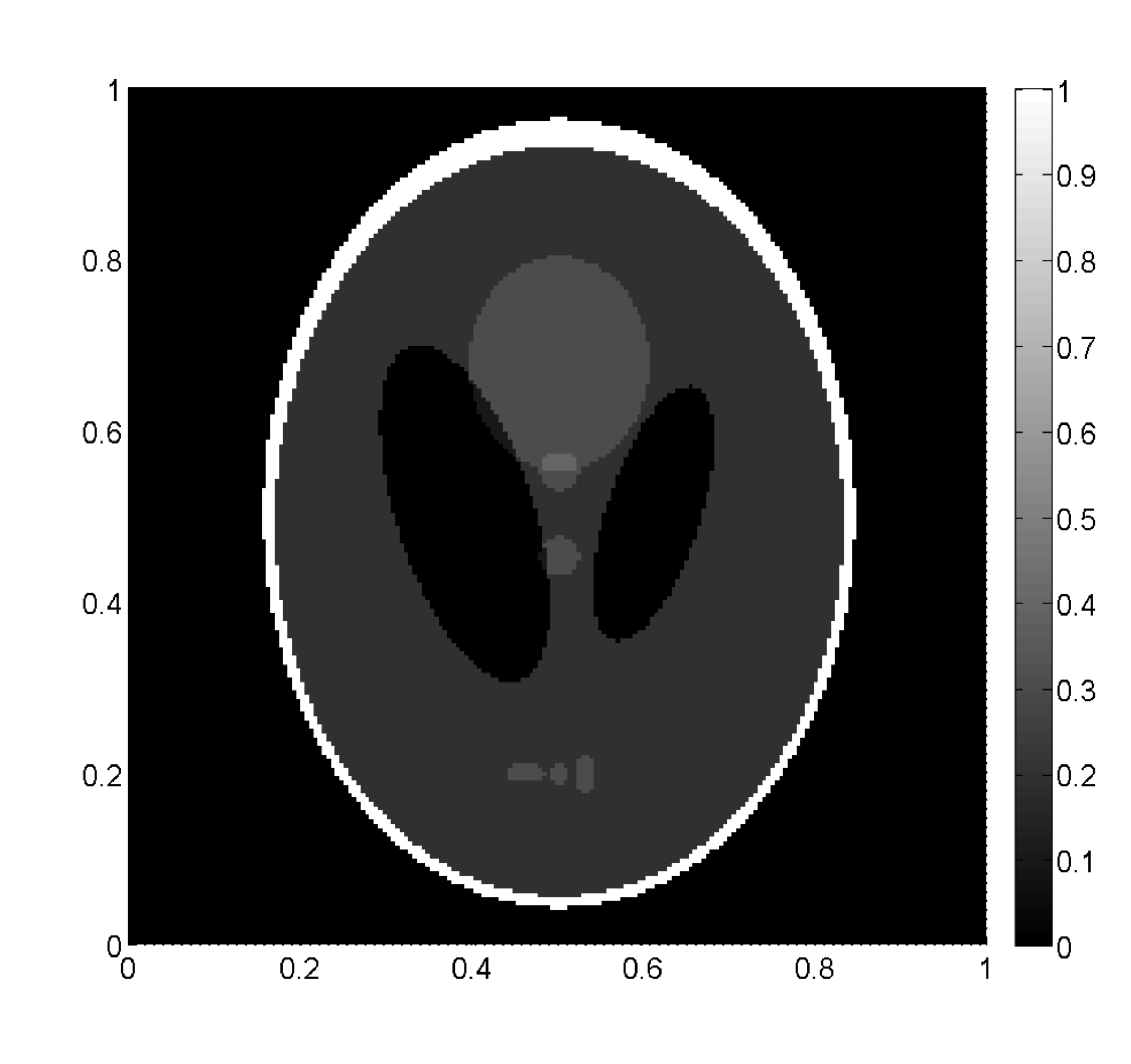} 
\includegraphics[height=0.33 \textheight, trim=10 10 50 10]{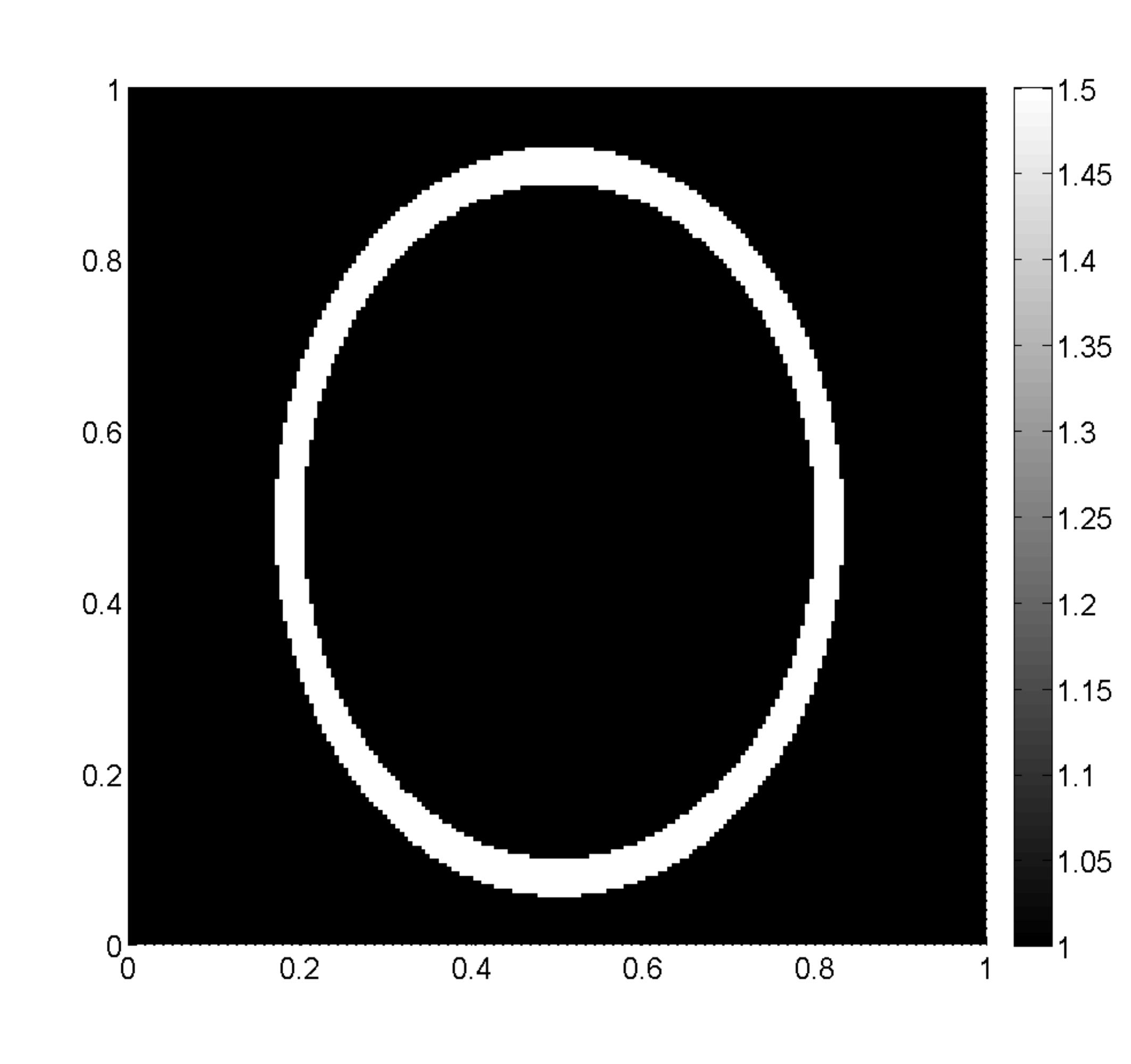} \\
\includegraphics[height=0.33 \textheight, trim=80 10 10 10]{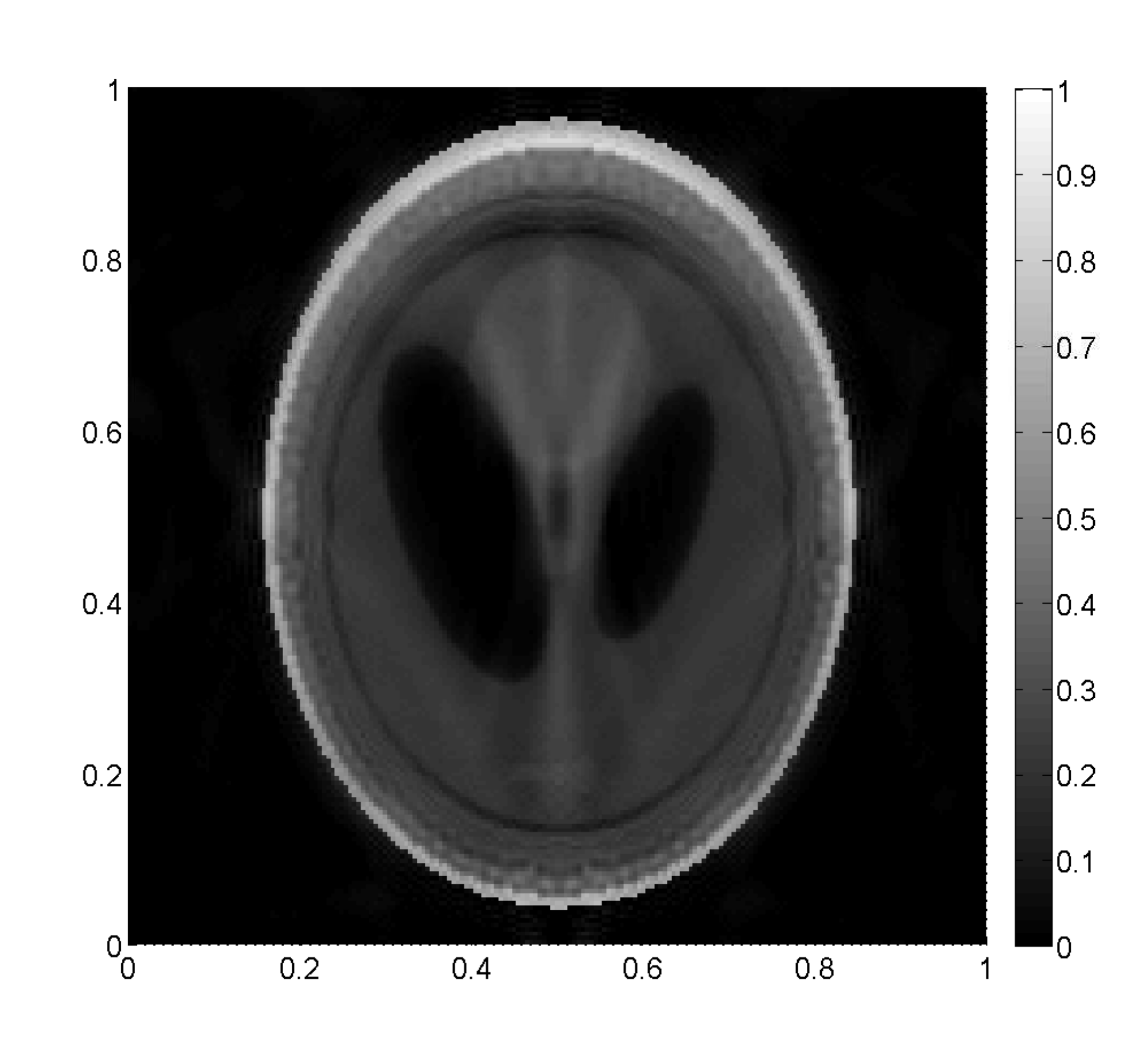} 
\includegraphics[height=0.33 \textheight, trim=10 10 50 10]{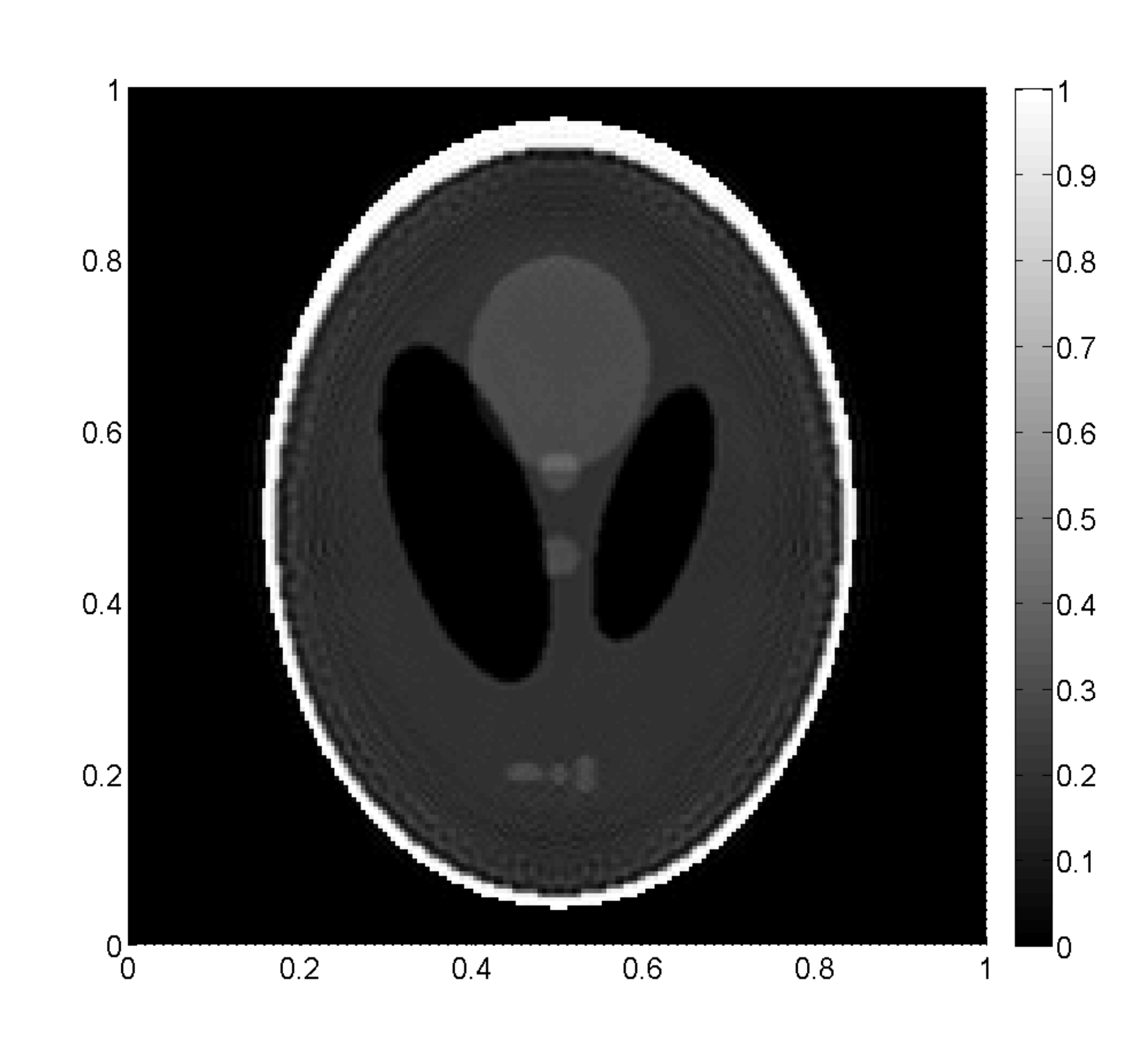} 
\caption{Exact initial acoustic profile (top--left), wave speed profile 
(top--right), the reconstruction from purely acoustic time--reversal (bottom--left) and the reconstruction from the proposed algorithm described in Section \ref{Section:Reconst} using $5$ iterations (bottom--right).}
\label{Fig.Comparison02}
\end{figure}

%%%%%%%%%%%%%%%%%%%%%%%%%%%%%%%%%%%%%%%%%%%%%%%%%%%%%%%%%%%%%%%%%%
%%%%%%%% NEW SECTION %%%%%%%%%%%%%%%%%%%%%%%%%%%%%%%%%%%%%%%%%%%%%
%%%%%%%%%%%%%%%%%%%%%%%%%%%%%%%%%%%%%%%%%%%%%%%%%%%%%%%%%%%%%%%%%%

\section{Conclusion} \label{Section:Discussion}

We have presented a PAT/TAT model based on thermoelasticity. The thermoelastic coupling accounts for how pressure changes can induce temperature changes in a body and vice versa. This coupling between temperature and deformation is a fundamental feature of PAT/TAT. The current literature dealing with PAT/TAT only considers one side of the thermoelastic interaction (the photoacoustic effect). By considering both effects simultaneously we account for a natural attenuation phenomenon.

We related the thermoelastic model of PAT/TAT with boundary observability for the thermacoustic system. We showed uniqueness and stability of recovering the initial pressure profile from boundary data provided that the thermoelastic coupling is weak. The recovery analysis of the initial wave profile is valid under a geometric assumption on the wave speed (see Assumption \ref{Assump.001}). We also proposed a reconstruction algorithm based on the conjugate gradient method. We carried out proof--of--concept numerical simulations to illustrate the implementation of the reconstruction algorithm for synthetic data. The authors are in the process of applying the proposed algorithm to actual experimental data. As soon as meaningful results are obtained from these efforts, they will be reported in a forthcoming publication.

For soft biological tissues, the unitless coupling  parameter $\epsilon$ of the thermoelastic model is approximately between 0.05 and 0.1 (as obtained from Table \ref{Table:Param}). Theorem \ref{Thm.MainResult} requires $\epsilon$ to be sufficiently small. Given that in PAT/TAT the thermodynamic interaction is small, such a condition on $\epsilon$ is reasonable. Nonetheless, it might be possible to remove this condition by using Carleman estimates for the coupled thermoelastic system (e.g., \cite{Eller2000,Albano-Tataru-2000,GLLT-2004}). The attenuation experienced by the shear waves has not been included in the present work. This may become relevant when the pressure waves interact with solid layers such as the skull \cite{Schoonover-Anas-2011,Huang-Nie-2012,Schoonover-2012}. In that case, it may be appropriate to incorporate the thermodynamic attenuation into the full elastic model of PAT/TAT \cite{Tittelfitz-2012}.

%% %%%%%%%%%%%%%%%%%%%%%%%%%%%%%%%%%%%%%%%%%%%%%%%%%%%%%%%%%%%%%%%%%%%%%%%%%%%%%%%

\section*{Acknowledgments} \label{Sec:Acknowledgements}
The authors would like to thank Plamen Stefanov for recommendations on the first draft of the paper. We also want to thank Benjam\'in Palacios for fruitful discussions and for exploring the possibility of a Neumann series approach for this problem.

%% %%%%%%%%%%%%%%%%%%%%%%%%%%%%%%%%%%%%%%%%%%%%%%%%%%%%%%%%%%%%%%%%%%%%%%%%%%%%%%%

\section*{References}

\bibliographystyle{unsrt}
\bibliography{Biblio}

\end{document}